\documentclass[12pt]{amsart}

\newtheorem{theorem}{Theorem}[section]
\newtheorem{lemma}[theorem]{Lemma}
\newtheorem{proposition}[theorem]{Proposition}
\newtheorem{corollary}[theorem]{Corollary}   
\newtheorem{definition}[theorem]{Definition}
\newtheorem{example}[theorem]{Example}

\newtheorem{remark}[theorem]{Remark}

\newtheorem{question}[theorem]{Question}

\numberwithin{equation}{section}

\newtheorem{algorithm}{Algorithm}

\usepackage{times}
\usepackage{enumerate}
\usepackage{mathrsfs}
\usepackage{tfrupee}
\usepackage{tikz}
\usetikzlibrary{chains,fit}
\usetikzlibrary{shapes,snakes}
\usetikzlibrary{graphs}
\usepackage{graphicx,adjustbox}
\usepackage{hyperref}
\usepackage{amsmath,amssymb}
\usepackage{amscd}
\usepackage{graphicx}
\usepackage[all]{xy}

\usepackage{float}
\usepackage{caption}
\usepackage{subcaption}

\usepackage{cleveref}

\begin{document}
\title[$\mathbf{CCM}$-completion]{closed Cohen-Macaulay completion of binomial edge ideals}
\author{
Kamalesh Saha
\and
Indranath Sengupta
}
\date{}

\address{\small \rm  Discipline of Mathematics, IIT Gandhinagar, Palaj, Gandhinagar, 
Gujarat 382355, INDIA.}
\email{kamalesh.saha@iitgn.ac.in}

\address{\small \rm  Discipline of Mathematics, IIT Gandhinagar, Palaj, Gandhinagar, 
Gujarat 382355, INDIA.}
\email{indranathsg@iitgn.ac.in}
\thanks{The second author is the corresponding author}

\date{}

\subjclass[2020]{Primary 05E40, 68Q25, 13H10, 13F65.}

\keywords{Completion, Binomial edge ideals, closed graphs, Cohen-Macaulay, unmixed.}

\allowdisplaybreaks

\begin{abstract}
Let $\mathbf{CCM}$ denote the class of closed graphs with 
Cohen-Macaulay binomial edge ideals and $\mathbf{PIG}$ denote the class of 
proper interval graphs. Then $\mathbf{CCM}\subseteq \mathbf{PIG}$. The 
$\mathbf{PIG}$-completion problem is a classical problem in molecular biology 
as well as in graph theory and this problem is known to be NP-hard. In this paper, 
we study the $\mathbf{CCM}$-completion problem. We give a method to 
construct all possible $\mathbf{CCM}$-completion of a graph. We find 
the $\mathbf{CCM}$-completion number and the set of all minimal 
$\mathbf{CCM}$-completions for a large class of graphs. Moreover, 
for that class, we give a polynomial-time algorithm to compute the 
$\mathbf{CCM}$-completion number and a minimum $\mathbf{CCM}$-completion 
of a given graph. We investigate unmixed and Cohen-Macaulay properties of 
binomial edge ideals of induced subgraphs. Also, we discuss the accessible 
graphs completion and the Cohen-Macaulay property of binomial edge ideals 
of whisker graphs.
\end{abstract}

\maketitle

\section{Introduction}

Two interesting connections between simple graphs and polynomial ideals 
are through edge ideals (see \cite{vil}) and binomial edge ideals (see \cite{hhhrkara}, 
\cite{oh}). Cohen-Macaulay graph ideals play a very important role in the field of 
combinatorial commutative algebra. Cohen-Macaulay edge ideals from arbitrary 
graphs can be constructed by adding some whiskers to the graph; see \cite{vi}. This has 
been generalized for $d$-uniform clutters through grafting; see \cite{F}. Some constructions 
of Cohen-Macaulay binomial edge ideals have been studied in \cite{raufrin}, using gluing 
of graphs and cone on graphs. Several classifications of Cohen-Macaulay and the unmixed property 
of binomial edge ideals have been investigated in \cite {mont}, \cite{bms}, \cite{bms1}, \cite{ehh}, \cite{km}, \cite{lmrr}, \cite{ms}, \cite{cactus}, \cite{rin_smaldev}. In the field 
of binomial edge ideals, the study of complexity of a problem is an area which is 
mostly untouched. In this paper, we study the complexity of 
constructions of closed Cohen-Macaulay binomial edge ideals.
\medskip

Let $\Pi$ be a class of graphs. A \textit{$\Pi$-completion} of a graph $G$ 
is a graph $H$ such that $V(H)=V(G)$, $E(G)\subseteq E(H)$, and $H\in \Pi$. The edges in 
$E(H)\setminus E(G)$ are called \textit{fill edges}. We write 
$\vert E(H)\setminus E(G)\vert=\Pi_{H}(G)$. A $\Pi$-completion $[G]$ of $G$ is said to be 
\textit{minimal} if there exists no $\Pi$-completion $H$ of $G$ such that 
$E(G)\subseteq E(H)\subsetneq E([G])$. A \textit{minimum $\Pi$-completion} of $G$ is a 
$\Pi$-completion $[G]$ of $G$ such that for any $\Pi$-completion $H$ of $G$, we have 
$\Pi_{H}(G)\geq \Pi_{[G]}(G)$. In this case, $\Pi_{[G]}(G)$ is called the 
\textit{$\Pi$-completion number} of $G$, and we denote it by $\Pi(G)$, i.e., 
the minimum number of edges required to add in $G$ to get a $\Pi$-completion of $G$ 
is $\Pi(G)$. The \textit{$\Pi$-completion problem} is to find $\Pi(G)$ of a given 
graph $G$. There are various types of completion problems in mathematics which have  
been tackled through graph theory, for example, the matrix completion problem (\cite{matcomp}) and 
the network completion problem (\cite{netcomp}). Graph completion problems like CHORDAL-completion, 
STRONGLY CHORDAL-completion, INTERVAL GRAPH- completion, have a rich history and motivations 
to study (see \cite{kaplan99}, \cite{rst06}). These problems have several practical 
applications in the fields like graph modelling (missing edges correspond to 
lack of data), molecular biology and computational algebra.
\medskip

The study of binomial edge ideals started through a particular class 
of graphs, known as \textit{closed graphs}. Later it was proved that a graph is closed if and only 
if it is a proper interval graph or $\mathbf{PIG}$ (see \cite{hho}). Among various 
graph completion problems, the study of $\mathbf{PIG}$-completion problem has become popular 
because of its presence in molecular biology, in particular, the structural 
study of DNA (see \cite{gks94} for details). Some work have been done in this direction in 
terms of the complexity of the problem (see \cite{dross21}, \cite{kaplan99}, \cite{rst06}). 
Since PIG-completion problem is known to be NP-hard (\cite{gks94}), it is 
common to study the complexity of $\Pi$-completion problem for some subclasses $\Pi$ of $\mathbf{PIG}$. 
Let us denote the class of closed graphs $G$ with Cohen-Macaulay $J_{G}$ by $\mathbf{CCM}$, where $J_{G}$ 
denotes the binomial edge ideal of $G$; clearly $\mathbf{CCM} \subseteq \mathbf{PIG}$. This paper is 
devoted to the study of complexity of the $\mathbf{CCM}$-completion problem.
\medskip

Our primary aim in this paper is to give a method to construct a graph $[G]$, 
which is a $\mathbf{CCM}$-completion of a given graph $G$, and 
study the complexity of this construction. In this construction, we choose an arbitrary 
labelling of the graph and keep adding suitable edges only. For an arbitrary graph, 
our construction can compute all possible $\mathbf{CCM}$-completion of that graph 
by taking all possible labelling. But, for a graph $G$ on $[n]=\{1,\ldots,n\}$, this 
method would take $n!$ iteration to compute $\mathbf{CCM}(G)$. The question that 
emerges from our construction is whether we can find $\mathbf{CCM}(G)$, and a 
minimum $\mathbf{CCM}$-completion $[G]$ of a given graph $G$ in polynomial time 
or not. Finally, for a large class of graphs, we give a polynomial-time algorithm 
to compute a minimum $\mathbf{CCM}$-completion and $\mathbf{CCM}$-completion number. 
We also study some unmixed and Cohen-Macaulay properties of binomial edge ideals in 
terms of subgraphs. The paper is arranged in the following manner.
\medskip

We first discuss some preliminaries in Section \ref{ccmpreli}, which 
are required for the subsequent discussions. The rest of the paper is divided into two parts. 
The first part (Sections \ref{ccm-completion}) is devoted to the $\mathbf{CCM}$-completion problem, 
and the second part (Section \ref{ccmsubgraph}) is devoted to the study of unmixed 
and Cohen-Macaulay properties of binomial edge ideals of induced subgraphs. To be precise, 
in Section \ref{ccm-completion}, we give a method to construct a graph $[G]$, which is a 
$\mathbf{CCM}$-completion of an arbitrary given graph $G$. The first step in the construction 
is to make the graph $G$ closed. Then we construct the graph $[G]$ by adding some special edges 
to $G$. By this process, we can find all possible $\mathbf{CCM}$-completion of a given graph 
by taking all possible labelling of the vertices. In Proposition \ref{closedpath}, we show that 
the block graph of a connected closed graph is a path, and in Proposition \ref{closedindecom}, 
we prove that an indecomposable closed graph has no cut vertex. We give the explicit structure 
of closed graphs with Cohen-Macaulay binomial edge ideals (see Theorem \ref{closedcmstruc}), 
which help us identify $\mathbf{CCM}$ graphs without knowing the labelling of the vertices. 
In Remark \ref{remarkclosedcm}, we discuss some properties of $J_{G}$ for $G\in \mathbf{CCM}$. 
Since the block graph of a graph with unmixed binomial edge ideal is a tree 
(by \cite[Proposition 1.3]{cactus}), it is worth considering the class of graphs whose block 
graphs are trees, and we denote this class by $\mathbf{BT}$. In Theorem \ref{smallclosedcm}, 
for a graph $G\in \mathbf{BT}$, we find the $\mathbf{CCM}$-completion number $\mathbf{CCM}(G)$, 
and the set of all minimal $\mathbf{CCM}$-completions of $G$. We 
write an algorithm (Algorithm \ref{algclosedcm}) to compute the $\mathbf{CCM}$-completion 
number $\mathbf{CCM}(G)$ and a 
minimum $\mathbf{CCM}$-completion $[G]$ of a given graph $G\in\mathbf{BT}$. We study the time 
complexity of the algorithm, and our algorithm turns out to be a polynomial-time algorithm. 
However, in general, it remains to check whether the $\mathbf{CCM}$-completion problem is NP-hard 
or not. So, we end this section by Question \ref{quesccm-comp}. In Section \ref{ccmsubgraph}, we 
try to give some conditions for unmixed (respectively, Cohen-Macaulay) property of binomial edge 
ideals of subgraphs of those graphs whose binomial edge ideal is unmixed 
(respectively, Cohen-Macaulay). In Theorem \ref{um9}, we show that for any closed graph $G$ 
if $J_{G}$ is Cohen-Macaulay, then $J_{H}$ is Cohen-Macaulay for any subgraph $H$ of $G$. 
We prove that any graph with unmixed binomial edge ideal has an accessible completion 
(Proposition \ref{acc-construct}), and propose the Question \ref{quesacc-comp}. We end up by 
giving the sufficient and necessary conditions (Theorem \ref{whiskercm}) for the Cohen-Macaulay 
property of binomial edge ideals of whisker graphs.

\section{Preliminaries}\label{ccmpreli}

All graphs are assumed to be simple. For a graph $G$, $V(G)$ denotes the vertex set of $G$, and $E(G)$ denotes the edge set of $G$. This paper uses the term subgraph to mean an induced subgraph of a graph.
\medskip

Let $V = \lbrace x_{1},\ldots, x_{n}\rbrace$. A \textit{simplicial complex} $\Delta$ on the vertex set $V$ is a collection
of subsets of $V$, with the following properties:
\begin{enumerate}
\item[(i)] $\lbrace x_{i}\rbrace\in \Delta$ for all $x_{i}\in V$;
\item[(ii)] $F\in \Delta$ and $G\subseteq F$ imply $G\in \Delta$.
\end{enumerate}

An element $F\in \Delta$ is called a \textit{face} of $\Delta$. A maximal face of $\Delta$ is called a \textit{facet} of 
$\Delta$. A vertex $i$ of $\Delta$ is called a 
\textit{free vertex} of 
$\Delta$ if $i$ belongs to exactly one facet.

\begin{definition}{\rm
A graph is said to be \textit{complete} if there is an edge between every pair of two vertices. Complete graph on $n$ vertices is denoted by $K_{n}$. A \textit{clique} of a given graph is a subset of vertex set on which the induced subgraph is complete.
}
\end{definition}
\begin{definition}{\rm 
The \textit{clique complex} $\Delta(G)$ of a graph $G$ is the 
simplicial complex whose faces are the cliques of $G$. 
Hence, a vertex $v$ of a graph $G$ is called 
\textit{free vertex} if it belongs to only one 
maximal clique of $\Delta(G)$.
}
\end{definition}

\begin{definition}{\rm
A \textit{path graph} $P$ on $n$ vertices is a graph whose vertices can be ordered as $v_{1},\ldots, v_{n}$ such that $E(P)=\{\{v_{i},v_{i+1}\}\mid 1\leq i\leq n-1\}$.
}
\end{definition}

\begin{definition}{\rm
A \textit{path} in a graph $G$ is a sequence of vertices $ v_{0},v_{1},\ldots, v_{s}$ such that $\{v_{i}, v_{i+1}\} \in E(G)$ for $i =0,\ldots,s-1$. If $G$ is labelled with $V(G)=[n]$, then a path $i_{0},\ldots,i_{s}$ is said to be \textit{directed} if either $i_{k} < i_{k+1}$ for all $k$ or $i_{k} > i_{k+1}$ for all $k$.
}
\end{definition}

Let $G$ be a graph and $v$ be a vertex of $G$. The \textit{neighbor set} of $v$ in $G$, denoted by $\mathcal{N}_{G}(v)$, is $\mathcal{N}_{G}(v)=\{u\in V(G)\mid \{u,v\}\in E(G)\}$. We write $\mathcal{N}_{G}[v]$ to denote the set $\mathcal{N}_{G}(v)\cup\{v\}$. Note that for a free vertex $v$, $\mathcal{N}_{G}[v]=F$, where $F$ is the facet of $\Delta(G)$ with $v\in F$. A vertex $v$ is said to be a \textit{cut vertex} or \textit{cut point} of $G$, if $G\setminus \{v\}$ has more number of connected components than $G$.

\begin{definition}{\rm 
Let $G$ be a graph on the vertex set $V(G)=[n]$ and $K$ be a field. The \textit{binomial edge ideal} $J_{G}\subseteq R=K[x_{1},\ldots,x_{n},y_{1},\ldots,y_{n}]$ of $G$ is the ideal generated by the binomials $f_{ij} = x_{i}y_{j}-x_{j}y_{i}$, such that $i<j$ and $\{ i, j\}\in E(G)$, where $E(G)$ denotes the edge set of $G$.
}
\end{definition}

Let us first recall some notations and results from \cite{hho} and \cite{raufrin}.

\begin{theorem}[{\cite[Theorem 7.2]{hho}}]\label{quadgb} Let $G$ be a graph on $[n]$, and let $<$ be the lexicographic order on $R=K[x_{1},\ldots,x_{n}, y_{1},\ldots,y_{n}]$, induced by $x_{1} > x_{2} >\cdots > x_{n} > y_{1} > y_{2} >\cdots > y_{n}$. The following conditions are equivalent:
\begin{enumerate}[(i)]
\item The generators $f_{ij}$ of $J_{G}$ form a quadratic 
Gr\"{o}bner basis;

\item For all edges $\lbrace i, j\rbrace$ and $\lbrace k, l\rbrace$ 
with $i < j$ and $k < l$, one has $\lbrace j, l\rbrace\in E(G)$ if $i = k$, 
and $\lbrace i, k\rbrace\in E(G)$ if $j = l$.
\end{enumerate}
\end{theorem}

A graph $G$ is said to be \textit{closed} with respect to a given labelling of vertices if it satisfies any of the equivalent conditions of Theorem \ref{quadgb}. By a \textit{closed} graph $G$, we mean a graph $G$ which is closed with respect 
to a suitable labelling of its vertices.

\begin{proposition}{\rm (\cite{hho}; Proposition 7.24)}\label{closedCM}
Let $G$ be a connected graph on $[n]$, which is closed with respect
to the given labelling. The following conditions are equivalent:
\begin{enumerate}[(i)]
\item $J_{G}$ is unmixed;
\item $J_{G}$ is Cohen-Macaulay;
\item $G$ satisfies the condition that whenever $\lbrace i, j + 1\rbrace$, 
with $i < j$ and $\lbrace j, k + 1\rbrace$, with $j < k$, are edges of $G$, then $\lbrace i, k + 1\rbrace$ is an edge of $G$.
\end{enumerate}
\end{proposition}

Let $T\subseteq [n]$, and $\overline{T}=[n]\setminus T$. Let $G[T]$ denotes the induced subgraph of $G$ on $T$. By $G\setminus T$, we denote the induced subgraph $G[\overline{T}]$. Let $c_{G}(T)$ (sometime we write $c(T)$ if the graph is clear from the context) denotes the number of connected components of the induced subgraph on $\overline{T}$, namely $G[\overline{T}]$ or $G\setminus T$. If each $i\in T$ is a cut point of the graph $G\setminus (T\setminus\{ i\})$, then we say that $T$ has \textit{cut point property} for $G$ or $T$ is a \textit{cutset} of $G$. We denote by $\mathscr{C}(G)$ the set of all cutsets of $G$.

\begin{proposition}[{\cite[Proposition 2.1]{raufrin}}]\label{um1}
Let $G$ be a graph, $\Delta(G)$ its clique complex and $v\in V(G)$. Then the following conditions are equivalent:
\begin{enumerate}[(i)]
\item There exists $T\in \mathscr{C}(G)$, such that $v\in T$.
\item $v$ is not a free vertex of $\Delta(G)$.
\end{enumerate}
\end{proposition}

\begin{lemma}[{\cite[Lemma 2.2]{raufrin}}]\label{um2}
Let $G$ be a graph, $v\in V(G)$, such that $v$ is a free vertex in $\Delta(G)$. 
Let $F$ be the facet of $\Delta(G)$, with $v\in F$, and $T\subseteq V(G)$ with 
$F\setminus\lbrace v\rbrace\not\subseteq T$. The following conditions are equivalent:
\begin{enumerate}[(i)]
\item $T\in\mathscr{C}(G)$.
\item $v\not\in T$ and $T\in\mathscr{C}(G\setminus\lbrace v\rbrace)$.
\end{enumerate}
\end{lemma}

\begin{lemma}[{\cite[Lemma 2.5]{raufrin}}]\label{um3}
Let $G$ be a connected graph. The following conditions are equivalent:
\begin{enumerate}[(i)]
\item $J_{G}$ is unmixed.
\item For all $T\in\mathscr{C}(G)$, we have $c(T)=\vert T\vert +1$.
\end{enumerate}
\end{lemma}

\begin{definition}[{\cite[Definition 2.2]{bms1}}]{\rm
Let $G$ be a simple graph. A cutset $T\in \mathscr{C}(G)$ is said to be \textit{accessible} if there exists $v\in T$ such that $T\setminus\{v\}\in \mathscr{C}(G)$. The graph $G$ is called \textit{accessible} if $J_{G}$ is unmixed and $T$ is accessible for all $T\in \mathscr{C}(G)$.
}
\end{definition}

\begin{definition}{\rm
Let $B_{1},\ldots, B_{r}$ be the blocks of a graph $G$. The \textit{block graph} of $G$, denoted by $\mathcal{B}(G)$, is the graph defined as follows:
\begin{enumerate}
\item[$\bullet$] $V(\mathcal{B}(G))=\{B_{1},\ldots, B_{r}\}.$
\item[$\bullet$] $E(\mathcal{B}(G))=\{\{B_{i},B_{j}\}\mid V(B_{i})\cap V(B_{j})\neq \phi\}$.
\end{enumerate}
\noindent By \cite[Proposition 1.3]{cactus}, the block graph of a connected graph with unmixed binomial edge ideal is a tree.
}
\end{definition}

For a graph $G$, the Cohen-Macaulay (resp. unmixed) property of $J_{G}$ 
is equivalent to the Cohen-Macaulay (resp. unmixed) property of the binomial edge ideals of its connected components. Therefore, we will assume all given graphs are connected unless otherwise stated.

\section{$\mathbf{CCM}$-completion Problem with Algorithm}\label{ccm-completion}
In this Section, from an arbitrary graph $G$ 
we construct a new graph $[G]$, by adding some edges to $G$, 
such that $[G]$ is a $\mathbf{CCM}$-completion of $G$. Also, 
we give an algorithm to find the $\mathbf{CCM}$-completion number 
for a large class of graphs and discuss its complexity.

\begin{definition}{\rm \label{intgraph}
A graph $G$ is called an \textit{interval graph} if for all $v\in V(G)$, 
there exists an interval $I_{v} = [l_{v}, r_{v}]$ of the real line such 
that $I_{v}\cap I_{w}\neq\phi$ if and only if $\lbrace v,w\rbrace\in E(G)$. 
If, in addition, the intervals can be chosen such that there is no 
proper containment among them, then $G$ is called a \textit{proper interval graph} 
or simply a \textit{PI} graph.
}
\end{definition}

Let $G$ be a graph. A set of intervals $\lbrace I_{v}\rbrace_{v\in V(G)}$, 
as in Definition \ref{intgraph}, is called an interval representation of G. 
Let $G$ be a graph on the vertex set $[n]$. Then $G$ satisfies the proper 
interval ordering with respect to the given labelling, if for all 
$i < j < k$, with $\lbrace i, k\rbrace\in E(G)$, it follows that 
$\lbrace i,j\rbrace, \lbrace j,k\rbrace\in E(G)$. We say that 
$G$ admits a \textit{proper interval ordering} if $G$ satisfies the proper 
interval ordering for a suitable relabelling of its vertices.

\begin{theorem}[{\cite[Theorem 1]{lo}}]\label{intordering}
A graph $G$ is a proper interval graph if and only if $G$ has a proper interval ordering.
\end{theorem}

\begin{theorem}[{\cite[Theorem 7.9]{hho}}]\label{propintgraph}
A graph $G$ is a closed graph if and only if $G$ is a proper interval graph.
\end{theorem}

\begin{remark}\label{pio}{\rm
From Theorem \ref{intordering} and Theorem \ref{propintgraph}, we can say that 
a graph $G$ is a closed graph if and only if $G$ has a proper interval ordering.
}
\end{remark}

\begin{corollary}\label{closedcriteria}
Suppose $G$ is a closed graph with respect to a suitable labelling of vertices. Let $G^{\prime}$ be the graph obtained by adding an edge $\lbrace i,k\rbrace$, 
with $i<k$ to $G$. Then $G^{\prime}$ is closed if and only if 
$\lbrace i,j\rbrace, \lbrace j,k\rbrace\in E(G^{\prime})$, for all 
$j$ with $i<j<k$.
\end{corollary}

\begin{proof}
We have $E(G^{\prime})=E(G)\cup \lbrace i,k\rbrace $. Since $G$ satisfies proper interval ordering 
(as $G$ is closed), for $\lbrace l,n\rbrace\in E(G)$ we have $\lbrace l,m\rbrace, \lbrace m,n\rbrace\in E(G)\subseteq E(G^{\prime})$, for all $l<m<n$. For $\lbrace i,k\rbrace\in E(G^{\prime})$ it follows from the given condition that 
$\lbrace i,j\rbrace, \lbrace j,k\rbrace\in E(G^{\prime})$, 
for all $j$ with $i<j<k$. So, $G^{\prime}$ satisfies the proper interval 
ordering and hence $G^{\prime}$ is closed.
\end{proof}

\begin{remark}\label{CMconstruct}
{\rm \textbf{(Construction of $[G]$ from $G$)} Start with an 
arbitrary connected graph $G$. If it is not closed then take 
any labelling of vertices and for all edges 
$\lbrace i, j\rbrace$, $\lbrace k, l\rbrace$ in $G$ 
with $i < j$, $k < l$, add $\lbrace j, l\rbrace$ to $E(G)$ if $i = k$, 
and add $\lbrace i, k\rbrace$ to $E(G)$ if $j = l$ and repeat the 
process with the newly obtained graph. Then after a finite number of 
steps, we get a new graph which is closed. Let us therefore 
assume that $G$ is a connected 
graph on $[n]$, which is closed with respect to a given 
labelling. We now construct a Cohen-Macaulay binomial edge ideal by adding 
some edges to $G$. If $\lbrace i,j+1\rbrace$, with $i < j$, and 
$\lbrace j, k + 1\rbrace$, with $j < k$ are edges of $G$, then we add 
the edge $\lbrace i, k + 1\rbrace$ to $G$ (if it is not there) and call 
the new graph $G'$. We repeat the process for $G'$. 
Since a finite graph has a finite number of edges, after finite 
number of steps, we would get a new graph $[G]$ on $[n]$ such that if $\lbrace i, j + 1\rbrace$ with $i < j$, 
and $\lbrace j, k + 1\rbrace$  
with $j < k$, are edges of $[G]$, then $\lbrace i, k + 1\rbrace\in E([G])$. Hence, if 
one can show that $[G]$ is closed, then 
by Proposition \ref{closedCM}, it follows that $J_{[G]}$ is Cohen-Macaulay 
and $J_{G}\subseteq J_{[G]}\subseteq K[x_{1},\ldots,x_{n},y_{1},\ldots,y_{n}]$. 
Therefore, starting with an arbitrary connected graph, we can construct a 
new graph whose corresponding binomial edge ideal is Cohen-Macaulay by 
adding some special edges. Note that the graph we obtain finally 
may not be unique; it depends on our choice of labelling the vertices 
to make it closed. Therefore, one has to label the vertices in 
such a way that one requires the least number of edges to make it closed.
}
\end{remark}

\begin{proposition}\label{closed}
The graph $G'$ constructed from $G$ by adding 
suitable edges (the procedure mentioned above) 
is closed. Hence, $[G]$ is closed.
\end{proposition} 

\begin{proof}
We prove that at the end of each step of adding all the necessary edges, 
the graph remains closed. Let us denote the graph by $G'$, after the first 
step. To prove $G'$ is closed, by Corollary \ref{closedcriteria}, it is enough to show that for each newly 
added edge $\lbrace i,k+1\rbrace$ in $G'$, the edges $\lbrace i,m\rbrace$ and $\lbrace m,k+1\rbrace$ 
belong to $G'$ for all $i<m<k+1$. Suppose that 
for the edges $\lbrace i,j+1\rbrace$, with $i < j$ and 
$\lbrace j, k + 1\rbrace$, with $j<k$, we add a new edge 
$\lbrace i, k + 1\rbrace$ to $G$. Now $\lbrace i,j+1\rbrace\in E(G)$ implies that $\lbrace i,i+1\rbrace,\lbrace i,i+2\rbrace,\ldots,\lbrace i,j\rbrace\in E(G)\subseteq E(G')$ and $\lbrace i+1,j+1\rbrace,\lbrace i+2,j+1\rbrace,\ldots,\lbrace j,j+1\rbrace\in E(G)\subseteq E(G')$, as $G$ is closed. Again 
$\lbrace j,k+1\rbrace\in E(G)$ implies that $\lbrace j,j+1\rbrace,\lbrace j,j+2\rbrace,\ldots,\lbrace j,k\rbrace\in E(G)\subseteq E(G')$ and $\lbrace j+1,k+1\rbrace,\lbrace j+2,k+1\rbrace,\ldots,\lbrace k,k+1\rbrace\in E(G)\subseteq E(G')$, as $G$ is closed. Now $\lbrace i,j+1\rbrace\in E(G)$ and $\lbrace j,j+2\rbrace\in E(G)$ imply that 
$\lbrace i,j+2\rbrace\in E(G')$, by our construction of $G'$. Similarly $\lbrace i,j+3\rbrace,\lbrace i,j+4\rbrace,\ldots,\lbrace i,k\rbrace\in E(G')$. So we have
\begin{equation}
\lbrace i,m\rbrace\in E(G')\quad \forall \, i<m<k+1.
\end{equation}
We already have $\lbrace j,k+1\rbrace,\lbrace j+1,k+1\rbrace,\lbrace j+2,k+1\rbrace,\ldots,\lbrace k,k+1\rbrace\in E(G)\subseteq E(G')$, from the closed property of $G$. 
If $i+1<j$, then $\lbrace i+1,j+1\rbrace\in E(G)$ and $\lbrace j,k+1\rbrace\in E(G)$ imply that $\lbrace i+1,k+1\rbrace\in E(G')$, by our construction of $G'$. Similarly $\lbrace i+2,k+1\rbrace,\lbrace i+3,k+1\rbrace,\ldots,\lbrace j-1,k+1\rbrace\in E(G')$. So we have
\begin{equation}
\lbrace m,k+1\rbrace\in E(G^{\prime})\quad \forall \, i<m<k+1.
\end{equation}
Therefore, by (2.1) and (2.2), for each newly added edge $\lbrace i,k+1\rbrace$ in $G'$, the edges $\lbrace i,m\rbrace$ and $\lbrace m,k+1\rbrace$ belong to $G'$, for all $i<m<k+1$. Hence $G^{\prime}$ is closed. Similarly, at each step of construction we 
get closed graph and therefore $[G]$ is closed.
\end{proof}

\begin{corollary}\label{closedCMconstruct}
Starting with an arbitrary connected graph $G$ and a labelling of $V(G)$, we can construct a $\mathbf{CCM}$-completion $[G]$ of $G$ by adding some suitable edges.
\end{corollary}

\begin{proof} 
Follows from Proposition \ref{closed} and Proposition \ref{closedCM}.
\end{proof}

\begin{example}{\rm
Let us illustrate the construction described above through an example. 
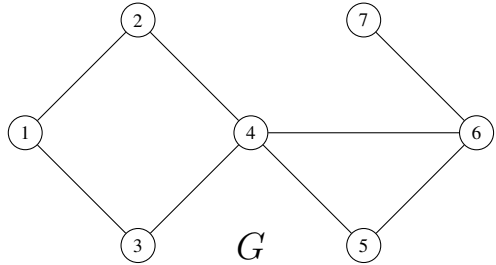
\begin{center}
\begin{tikzpicture}
  [scale=.5,auto=left,every node/.style={circle,scale=0.6}]
  \node[draw] (n1) at (0,0) {1};
  \node[draw] (n2) at (3,3)  {2};
  \node[draw] (n3) at (3,-3)  {3};
  \node[draw] (n4) at (6,0) {4};
   \node[draw] (n5) at (9,-3) {5};
  \node[draw] (n6) at (12,0)  {6};
  \node[draw] (n7) at (9,3)  {7};

  \foreach \from/\to in {n1/n2,n2/n4,n3/n4,n1/n3,n4/n5,n5/n6,n6/n7,n4/n6}
    \draw (\from) -- (\to);
    
    \node[fill=white, scale=2] at (6,-3) {$G$};
\end{tikzpicture}
\end{center}

\noindent The graph $G$ is not closed with respect to any labeling. In our choice of labelling, it is not closed because of $\lbrace 1,2\rbrace,\lbrace 1,3\rbrace\in E(G)$ but $\lbrace 2,3\rbrace\not\in E(G)$. So add $\lbrace 2,3\rbrace$ to make it closed and call the new graph $G^{\prime}$.
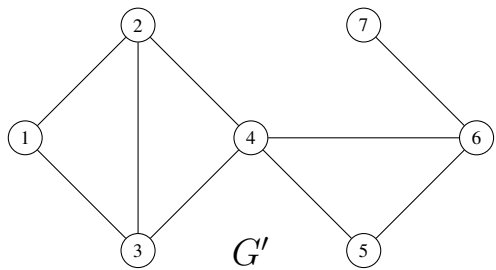
\medskip

\begin{center}
\begin{tikzpicture}
[scale=.5,auto=left,every node/.style={circle,scale=0.6}]
  \node[draw] (n1) at (0,0) {1};
  \node[draw] (n2) at (3,3)  {2};
  \node[draw] (n3) at (3,-3)  {3};
  \node[draw] (n4) at (6,0) {4};
   \node[draw] (n5) at (9,-3) {5};
  \node[draw] (n6) at (12,0)  {6};
  \node[draw] (n7) at (9,3)  {7};

  \foreach \from/\to in {n1/n2,n2/n4,n3/n4,n1/n3,n4/n5,n5/n6,n6/n7,n2/n3,n4/n6}
    \draw (\from) -- (\to);
    
    \node[fill=white, scale=2] at (6,-3) {$G^{\prime}$};
\end{tikzpicture}
\end{center}

\noindent Now $J_{G^{\prime}}$ is not Cohen-Macaulay as $\lbrace 1,3\rbrace,\lbrace 2,4\rbrace\in E(G)$ but $\lbrace 1,4\rbrace\not\in E(G)$. Add $\lbrace 1,4\rbrace$ 
and call the new graph $[G]$.
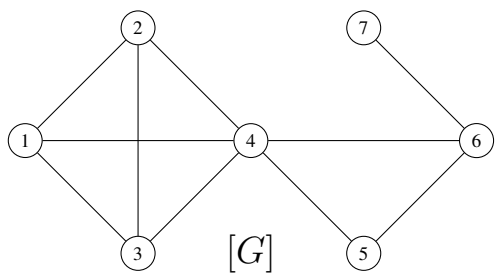
\medskip

\begin{center}
\begin{tikzpicture}
  [scale=.5,auto=left,every node/.style={circle,scale=0.6}]
  \node[draw] (n1) at (0,0) {1};
  \node[draw] (n2) at (3,3)  {2};
  \node[draw] (n3) at (3,-3)  {3};
  \node[draw] (n4) at (6,0) {4};
   \node[draw] (n5) at (9,-3) {5};
  \node[draw] (n6) at (12,0)  {6};
  \node[draw] (n7) at (9,3)  {7};

  \foreach \from/\to in {n1/n2,n2/n4,n3/n4,n1/n3,n4/n5,n5/n6,n6/n7,n2/n3,n4/n6,n1/n4}
    \draw (\from) -- (\to);
    
    \node[fill=white, scale=2] at (6,-3) {$[G]$};
\end{tikzpicture}
\end{center}

\noindent $[G]$ is closed and $J_{[G]}$ is Cohen-Macaulay by Proposition \ref{closedCM}.
}
\end{example}
\medskip

The construction described above depends on the labelling of the vertices, and to find 
the $\mathbf{CCM}$-completion number $\mathbf{CCM}(G)$ of a graph $G$, we have to check 
$n!$ cases, which cannot be executed in polynomial time. Next, we give the explicit structure 
of a closed Cohen-Macaulay binomial edge ideals in Theorem \ref{closedcmstruc}, which helps  
us identify the closed Cohen-Macaulay binomial edge ideals irrespective of labelling. 
Let $\mathbf{BT}$ denotes the class of those graphs whose block graphs are trees, i.e., 
$\mathbf{BT}=\{G\mid \mathcal{B}(G)\,\,\text{is\,\,a\,\,tree}\}$. For any graph 
$G\in \mathbf{BT}$, we give an algorithm (Algorithm \ref{algclosedcm}) to find the 
$\mathbf{CCM}$-completion number $\mathbf{CCM}(G)$ in polynomial time.
\medskip

\begin{proposition}\label{closedpath}
The block graph $\mathcal{B}(G)$ of a connected closed graph $G$ is a path graph.
\end{proposition}

\begin{proof}
Let $B_{l},B_{m},B_{n}$ be three blocks of $G$ such that $B_{l}\cap B_{m}\cap B_{n}=\{v\}$. 
There exist $u_{t}\in V(B_{t})$, such that $\{v,u_{t}\}\in E(G)$ for $t\in \{l,m,n\}$. Note that $G[T]$ 
is an induced claw for $T=\{v,u_{l},u_{m},u_{n}\}$ and by \cite[Proposition 7.4]{hho}, we get a contradiction 
as $G$ is closed. Therefore, no three blocks of $G$ can have a common vertex. Let $B$ be a block of $G$ 
containing more than two cut vertices of $G$. Consider a labelling of $G$ with respect to which $G$ is 
closed and by Remark \ref{pio}, $G$ admits a proper interval ordering with respect to this labelling. 
Let $i<j<k$ be three cut vertices of $G$ chosen arbitrarily and belonging to $V(B)$. Let $B_{i},B_{j},B_{k}$ 
be three blocks other than $B$ containing $i,j,k$, respectively. By \cite[Proposition 2.1]{closed1}, 
there exists a directed path from $i$ to $j$. Since $i<j$, there exists a vertex 
$j^{\prime}\in \mathcal{N}_{B}(j)$ such that $j^{\prime}<j$. By proper interval ordering of 
$G$, we have $\{j-1,j\}\in E(B)$ and $\{j,j+1\}\in E(B_{j})$. Therefore, no vertex of $\mathcal{N}_{B}(j)$ 
can be greater than $j$, otherwise $j+1$ will be adjacent to a 
vertex in $V(B)\setminus\{j\}$, which is not possible. Thus, there can not be any directed path from $j$ to $k$, 
as $j<k$, and this gives a contradiction due to \cite[Proposition 2.1]{closed1}. So, every block of $G$ contains 
at most two cut vertices of $G$. Moreover, $G$ being connected, there are exactly two blocks of $G$ each of which 
contains only one cut vertex of $G$. Hence the block graph $\mathcal{B}(G)$ of $G$ is a path graph.
\end{proof}

Let $G$ be a closed graph. Then by Proposition \ref{closedpath}, $G$ can be written as
\begin{align}\label{closeddecom}
G=B_{1}\cup\cdots\cup B_{r},
\end{align}
where $B_{1},\ldots, B_{r}$ are blocks of $G$ with the following property: For $i<j$ and $i\in\{1,\ldots,r-1\}$, 
we have $V(B_{i})\cap V(B_{j})\neq \phi$ if and only if $j=i+1$.

\begin{proposition}\label{closedindecom}
An indecomposable closed graph has no cut vertex. 
\end{proposition}

\begin{proof} Let $G$ be a closed graph with a cut vertex $v$. We show that $G$ is decomposable. 
By Proposition \ref{closedpath}, $v$ belongs to exactly two blocks of $G$, say $B_{1}$ and $B_{2}$. 
If $\mathcal{N}_{B_{i}}(v)$ is singleton for $i=1,2$, then we are done. Let $x,y\in \mathcal{N}_{B_{1}}(v)$ 
and $z\in \mathcal{N}_{B_{2}}(v)$ be any vertex. Suppose $\{x,y\}\not\in E(G)$. Then the induced subgraph 
$G[\{x,y,z\}]$ is a claw and this gives a contradiction by \cite[Theorem 7.10]{hho} to the fact that $G$ 
is closed. Thus, $\mathcal{N}_{B_{1}}(v)$ is a clique and similarly, $\mathcal{N}_{B_{2}}(v)$ is a clique. 
Hence $G$ is decomposable.

\end{proof}

\begin{theorem}\label{closedcmstruc} A graph $G$ is closed and $J_{G}$ is 
Cohen-Macaulay (equivalently, unmixed) if and only if every block of $G$ 
is complete and $\mathcal{B}(G)$ is a path graph.
\end{theorem}

\begin{proof}
Let $G$ be closed and $J_{G}$ be Cohen-Macaulay. Consider a labelling of $G$ with respect to 
which $G$ is closed. By Proposition \ref{closedpath}, $\mathcal{B}(G)$ is a path. Let $B$ be a block of $G$ 
containing two cut vertices $i$ and $j$ of $G$ with $i<j$. Let $B_{i}$ and $B_{j}$ be two blocks other than 
$B$ containing $i$ and $j$ respectively. Since $J_{G}$ is Cohen-Macaulay, $G$ is accessible by 
\cite[Theorem 3.5]{bms1}. Therefore, by \cite[Proposition 4.10]{bms1}, we have $\{i,j\}\in E(B)$. 
Since $i<j$ and $\{i,j\}\in E(B)$, proper interval ordering of $G$ implies 
$\{i,i+1\}, \{j-1,j\}\in E(B)$, $\{i-1,i\}\in E(B_{i})$ and $\{j,j+1\}\in E(B_{j})$. Thus, for 
any vertex $v\in V(B)$ we have $i\leq v\leq j$. Again, using proper interval ordering, we can 
say $B$ is complete as $\{i,j\}\in E(B)$. Now consider a block $B^{\prime}$ of $G$ containing 
only one cut vertex $k$ of $G$. By proper interval ordering of $G$, all vertices in 
$\mathcal{N}_{B^{\prime}}(k)$ are either less than $k$ or greater than $k$. Since $G$ is 
accessible, by \cite[Theorem 1.2]{bms1}, we have $\mathcal{N}_{B^{\prime}}[k]=V(B^{\prime})$. 
Hence, it follows from the proper interval property of $G$ that $B^{\prime}$ is complete. 
Conversely, if every block of $G$ is complete and $\mathcal{B}(G)$ is a path, then $G$ is 
closed by \cite[Theorem 7.10]{hho} and $J_{G}$ is Cohen-Macaulay by \cite[Theorem 1.1]{ehh}.
\end{proof}

\begin{remark}\label{closedcmremark}{\rm 
Let $G$ be a closed graph such that $J_{G}$ is Cohen-Macaulay. Then choosing a closed labelling of $G$, we can write $G$ as
$$ G=K_{m_{1}}\cup\cdots\cup K_{m_{r}},$$
where $V(K_{m_{i}})=\{m_{1}+\cdots+m_{i-1}-i+1+k \mid 1\leq k\leq m_{i}\} $ for $i=2,\ldots,r$ and $V(K_{m_{1}})=\{1,\ldots,m_{1}\}$.}
\end{remark}

Let $G$ be a simple graph. We denote the \textit{Castelnuovo-Mumford regularity} of $R/J_{G}$ 
by $\mathrm{reg}(R/J_{G})$, the \textit{Hilbert series} of $R/J_{G}$ by $\mathrm{Hilb}_{R/J_{G}}(t)$, 
the $i$'th \textit{Betti number} by $\beta_{i}(R/J_{G})$ and the $(i,j)$-th \textit{graded Betti number} 
by $\beta_{i,j}(R/J_{G})$.

\begin{remark}\label{remarkclosedcm}{\rm 
The following are some properties of closed graphs with Cohen-Macaulay binomial 
edge ideals. Let $G=K_{m_{1}}\cup\cdots\cup K_{m_{r}}$ be any closed graph such that $J_{G}$ is 
Cohen-Macaulay, as in Remark \ref{closedcmremark}, 
and $R_{i}=K[\{x_{j},y_{j}\mid j\in V(K_{m_{i}})\}]$ for $1\leq i\leq r$.
 
\begin{enumerate}[(1)]
\item By \cite[Theorem 1.1]{mm13}, we have $\mathrm{reg}(R/J_{H})\geq d(H)$, 
where $d(H)$ is the diameter of $H$, i.e., the length of the longest induced path in $H$. 
By \cite[Theorem 3.2]{km12}, we have $\mathrm{reg}(R/J_{H})\leq \mathrm{clq}(H)$ for any 
closed graph $H$, where $\mathrm{clq}(H)$ denotes the number of maximal cliques of $H$. Now, Theorem \ref{closedcmstruc} implies $d(G)=\mathrm{clq}(G)$ and so, we get
$$\mathrm{reg}(R/J_{G})=d(G)=\mathrm{clq}(G)=r.$$

\item Using \cite[Lemma 2.9]{zz00} and \cite[Proposition 3]{hr18}, we have 
\begin{align*}
\beta_{q}(R/J_{G})&= \sum_{i_{1}+\cdots+i_{r}=q}\beta_{i_{1}}(R_{1}/J_{K_{m_{1}}})\cdots \beta_{i_{r}}(R_{r}/J_{K_{m_{r}}})\\
&= \sum_{i_{1}+\cdots+i_{r}=q} i_{1}\binom{m_{1}}{i_{1}+1}\cdots i_{r}\binom{m_{r}}{i_{r}+1},
\end{align*}
 where $1\leq q\leq n-1$ and $\beta_{0}(R/J_{G})=1$.
 
\item From \cite[Corollary 1]{ch94}, we get $\mathrm{Hilb}_{R_{i}/J_{K_{m_{i}}}}(t)=\dfrac{1+(m_{i}-1)t}{(1-t)^{m_{i}+1}}$, for $1\leq i\leq r$. Hence by \cite[Corollary 3.3]{ks19}, we have
\begin{align*}
\mathrm{Hilb}_{R/J_{G}}(t)&=(1-t)^{2r-2}\prod_{i\in [r]} \mathrm{Hilb}_{R_{i}/J_{K_{m_{i}}}}(t)\\
&=(1-t)^{2r-2} \prod_{i\in [r]} \dfrac{1+(m_{i}-1)t}{(1-t)^{m_{i}+1}}\\
&=\dfrac{\prod_{i\in [r]} (1+(m_{i}-1)t)}{(1-t)^{n+1}}.
\end{align*}

\item By \cite[Corollary 7]{hr18}, $\beta_{n-1, n+2r-2}(R/J_{G})$ is an extremal Betti number of $R/J_{G}$ and $\beta_{n-1, n+2r-2}(R/J_{G})=\prod_{i\in [r]} (m_{i}-1)$.
\end{enumerate}
}
\end{remark}

Let $T$ be a tree. There exists one and only one path between any two vertices of $T$. 
Consider a path $u=u_{0}, u_{1},\ldots, u_{n}=v$ in $T$ between $u$ and $v$. Let $C^{uv}_{0}$ and $C^{uv}_{n}$ 
denote the connected components of $T\setminus\{u_{1}\}$ containing $u_{0}$ and $T\setminus\{u_{n-1}\}$ 
containing $u_{n}$, respectively. Let $C^{uv}_{i}$ denote the connected component of 
$T\setminus\{u_{i-1},u_{i+1}\}$ containing the vertex $u_{i}$ for $i=\{1,\ldots,n-1\}$. 
Then $V(C^{uv}_{0}),\ldots, V(C^{uv}_{n})$ creates a partition on $V(T)$.
\medskip
 
Let $G$ be a connected graph such that no three blocks of $G$ share a common vertex, i.e., 
$\mathcal{B}(G)$ is a tree. Let $B_{i}$ and $B_{j}$ be two blocks of $G$. Then $B_{i}$ and $B_{j}$ are two 
vertices in $\mathcal{B}(G)$. Since $\mathcal{B}(G)$ is a tree, there exists a unique path between $B_{i}$ 
and $B_{j}$ in $\mathcal{B}(G)$. Let $B_{i}=B^{ij}_{0}, B^{ij}_{1},\ldots, B^{ij}_{s-1}, B^{ij}_{s}=B_{j}$ 
be the unique path between $B_{i}$ and $B_{j}$ in $\mathcal{B}(G)$. Then $V(C^{ij}_{0}),\ldots, V(C^{ij}_{s})$ 
makes a partition on $V(\mathcal{B}(G))$. Now consider $D^{ij}_{k}=\cup_{B\in V(C^{ij}_{k})} V(B)$, as a 
subset of $V(G)$ for $k=\{0,\ldots,s\}$. Note that $D^{ij}_{0}=V(B_{i})$ and $D^{ij}_{s}=V(B_{j})$ as $i,j$ 
are pendant vertices of $G$. Let 
$$ [G]_{ij}= K_{\vert D^{ij}_{0}\vert}\cup\cdots\cup K_{\vert D^{ij}_{s}\vert},$$
where $V(K_{\vert D^{ij}_{k}\vert})= D^{ij}_{k}$ for $k=0,\ldots,s$. Then $[G]_{ij}$ is a closed graph such that $J_{[G]_{ij}}$ is Cohen-Macaulay and $J_{G}\subseteq J_{[G]_{ij}}$. We write 
$$d_{ij}=\mathbf{CCM}_{[G]_{ij}}(G)=\vert E([G]_{ij})\vert- \vert E(G)\vert.$$
Note that $[G]_{ij}=[G]_{ji}$, and $\vert E([G]_{ij})\vert =\sum_{k=0}^{s}\binom{\vert D^{ij}_{k}\vert}{2}$.

\begin{theorem}\label{smallclosedcm}
Let $G\in \mathbf{BT}$ be a graph. The set of all minimal completion of 
$G$ is $\{[G]_{ij}\mid B_{i},B_{j} \,\,\text{pendant\,\, vertices\,\, of}\,\, \mathcal{B}(G)\}$. 
Moreover,
$$\mathbf{CCM}(G)=\mathrm{min}\{d_{ij}\mid B_{i},B_{j} \,\,\text{pendant\,\, vertices\,\, of}\,\, \mathcal{B}(G)\},$$
and if $\mathbf{CCM}(G)=d_{pq}$, then $[G]_{pq}$ is a minimum $\mathbf{CCM}$-completion of $G$.
\end{theorem}

\begin{proof} By Theorem \ref{closedcmstruc}, any closed graph with Cohen-Macaulay binomial edge ideal is 
such that every block is complete and block graph is a path as mentioned in Remark \ref{closedcmremark}. 
Suppose $[G]$ is a closed graph with Cohen-Macaulay binomial edge ideal such that $V([G])=V(G)$ and 
$E(G)\subseteq E([G])$. Then by Theorem \ref{closedcmstruc}, we can write
 $$[G]=K_{m_{1}}\cup\cdots\cup K_{m_{r}},$$
where for $i<j$, $V(K_{m_{i}})\cap V(K_{m_{j}})=\phi$ if $j\neq i+1$ and $V(K_{m_{i}})\cap V(K_{m_{i+1}})$ 
is a singleton set containing a cut vertex. Note that there is a block $B_{m_{i}}$ of $G$ such that 
$V(B_{m_{i}})\subseteq V(K_{m_{i}})$ and $B_{m_{i}}$ contains exactly one cut vertex of $G$ for $i\in\{1,r\}$. 
Now consider the graph $[G]_{m_{1}m_{r}}$. Then $E([G]_{m_{1}m_{r}})\subseteq E([G])$ is clear. Also, 
observe that $[G]_{ij}$ is a minimal $\mathbf{CCM}$-completion of $G$ for every pair of pendant vertices 
$\{B_{i},B_{j}\}$ of $\mathcal{B}(G)$.  Therefore, $\{[G]_{ij}\mid B_{i},B_{j} \,\,\text{pendant\,\, vertices\,\, of}\,\, \mathcal{B}(G)\}$ is the collection of all minimal $\mathbf{CCM}$-completion of $G$. Since any 
minimum $\mathbf{CCM}$-completion is a minimal $\mathbf{CCM}$-Completion, we have
 $$\mathbf{CCM}(G)=\mathrm{min}\{d_{ij}\mid B_{i},B_{j} \,\,\text{pendant\,\, vertices\,\, of}\,\, \mathcal{B}(G)\},$$
and if $\mathbf{CCM}(G)=d_{pq}$, then $[G]_{pq}$ is a minimum $\mathbf{CCM}$-completion of $G$.
\end{proof}

\begin{corollary}
Let $G$ be a closed graph. Then the $\mathbf{CCM}$-completion number of $G$ is given by
$$\mathbf{CCM}(G)=\sum_{B\,\, \text{block\,\, of}\,\, G} \binom {\vert V(B)\vert}{2} - \vert E(G)\vert. $$
\end{corollary}

\begin{proof}
By Proposition \ref{closedpath}, $\mathcal{B}(G)$ is a path graph. So, $G$ can be decomposed as in \ref{closeddecom}. There are only two pendant vertices of $\mathcal{B}(G)$, say $B_{p}$ and $B_{q}$. Thus, by Theorem \ref{smallclosedcm}, the only minimal $\mathbf{CCM}$-completion of $G$ will be $[G]_{pq}$, and so, it is the only minimum $\mathbf{CCM}$-completion. Hence, $\mathbf{CCM}(G)=d_{pq}$. Since complete graph on the vertex set $V(B)$ of a block $B$ of $G$ contains $\binom {\vert V(B)\vert}{2}$ edges, the result follows easily.
\end{proof}

We give the following algorithm to find the $\mathbf{CCM}$-completion number $\mathbf{CCM}(G)$ and a minimum $\mathbf{CCM}$-completion of $G$, for any graph $G\in \mathbf{BT}$.

\begin{algorithm}\label{algclosedcm}{\rm
\textbf{Input:} A graph $G$ such that $\mathcal{B}(G)$ is a tree.
\medskip

\textbf{Output:} $\mathbf{CCM}(G)=d$ and a minimum $\mathbf{CCM}$-completion $[G]_{pq}$ of $G$.
\medskip

\begin{enumerate}[Step-1:]
\item Find the block graph $\mathcal{B}(G)$ of $G$.
\medskip

\item Collect all the pendant vertices of $\mathcal{B}(G)$ and let $\mathcal{P}=\{B_{1},\ldots, B_{r}\}$ be 
the set of all pendant vertices of $\mathcal{B}(G)$.
\medskip

\item Set $d=n^2, p=0,q=0$.
\medskip

\item Set $T=\{(B_{i},B_{j})\mid B_{i},B_{j}\in \mathcal{P}\,\, \text{with}\,\, i<j\}$ i.e., $T$ is the collection of all ordered pairs $(B_{i},B_{j})$ such that $B_{i},B_{j}$ are pendant vertices of $\mathcal{B}(G)$ with $i<j$.
\medskip

\item Take an element $(B_{i},B_{j})\in T$.
\medskip

\item Find the unique path from $B_{i}$ to $B_{j}$ in $\mathcal{B}(G)$ and suppose the path is $B_{i}=B^{ij}_{0}, \ldots, B^{ij}_{s_{ij}}=B_{j}$.
\medskip

\item Set $C^{ij}_{0}=$ the connected component of $\mathcal{B}(G)\setminus\{B^{ij}_{1}\}$ containing $B^{ij}_{0}$ i.e., $V(C^{ij}_{0})$ is the collection of those vertices of $\mathcal{B}(G)$ which are connected to $B^{ij}_{0}$ via a path in $\mathcal{B}(G)\setminus\{B^{ij}_{1}\}$. Similarly, we set $C^{ij}_{s_{ij}}=$ the connected component of $\mathcal{B}(G)\setminus\{B^{ij}_{s_{ij}-1}\}$ containing $B^{ij}_{s_{ij}}$, and for $1\leq k\leq s_{ij}-1$,
$C^{ij}_{k}=$ the connected component of $\mathcal{B}(G)\setminus\{B^{ij}_{k-1}, B^{ij}_{k+1}\}$ containing $B^{ij}_{k}$.
\medskip

\item For each $k\in\{0,\ldots,s_{ij}\}$, consider
$$D^{ij}_{k}=\bigcup_{B\in V(C^{ij}_{k})} V(B)$$ as a subset of $V(G)$.
\medskip

\item Compute $$d_{ij}=\sum_{k=0}^{s_{ij}}\binom{\vert D^{ij}_{k}\vert}{2}- \vert E(G)\vert.$$
\medskip

\item If $d_{ij}<d$, then set $d=d_{ij}, p=i, q=j$.
\medskip

\item Update $T=T\setminus\{(B_{i},B_{j})\}$.
\medskip

\item If $T\neq \phi$, then go to Step-5, else go to Step-13.
\medskip

\item Construct the graph $[G]_{pq}$ with 
$$V([G]_{pq})=V(G),$$
 and 
 $$E([G]_{pq})=\bigcup_{k=0}^{s_{pq}} \{\{u,v\}\mid u,v\in D^{pq}_{k}, u\neq v\}.$$
 
\item Output $d$ and $[G]_{pq}$.

\end{enumerate}
}
\end{algorithm}
\medskip

\noindent\textbf{Correctness:} The correctness of the algorithm follows from Theorem \ref{smallclosedcm}.\medskip

\noindent\textbf{Time complexity:} Let the given graph $G$ has $V(G)=[n]$. 
Then by the algorithm given in \cite{paton}, we can compute $\mathcal{B}(G)$ in $\mathcal{O}(n^2)$ time. 
Now, checking whether a vertex of $\mathcal{B}(G)$ is a pendant vertex or not takes 
$\mathcal{O}(\vert V(\mathcal{B}(G))\vert ^{2})$ time. Therefore, all the pendant vertices of 
$\mathcal{B}(G)$ can be found in Step-2 by $\mathcal{O}(n^2)$ time (because 
$\vert V(\mathcal{B}(G))\vert\leq n$). Step-3 requires constant time. Since the number of pendant vertices 
of $\mathcal{B}(G)$ is $r$ and $\binom{r}{2}<n^2$, Step-4 can be executed in $\mathcal{O}(n^2)$ time. 
Step-5 takes $\mathcal{O}(1)$ time. We can find path between two vertices in $\mathcal{B}(G)$ using either the 
BFS or the DFS method (see \cite{bfsdfs}) in $\mathcal{O}(\vert V(\mathcal{B}(G))\vert+ \vert E(\mathcal{B}(G))\vert)$ 
time and thus, $\mathcal{O}(n^2)$ time is sufficient for Step-6. Since $\vert V(\mathcal{B}(G))\vert \leq n$, 
we can find any induced subgraph of $\mathcal{B}(G)$ in $\mathcal{O}(n^2)$ time. In any induced subgraph 
of $\mathcal{B}(G)$, we can find the connected component containing one fixed vertex in $\mathcal{O}(n^2)$ 
time using the DFS. Now in Step-7, we have to do this process $s_{ij}+1$ times. Therefore, Step-7 requires 
$\mathcal{O}(n^3)$ time as $s_{ij}+1\leq n$. Clearly, we have 
$\vert V(C^{ij}_{k})\vert\leq \vert V(\mathcal{B}(G))\vert\leq n$, and so, we can find $D^{ij}_{k}$ in 
$\mathcal{O}(n)$ time for each $k\in \{0,\ldots,s_{ij}\}$. Since $s_{ij}+1\leq n$, total time required 
to process Step-8 is $\mathcal{O}(n^2)$. In Step-9, we are taking sum of $s_{ij}+1$ terms and so it will 
take $\mathcal{O}(n)$ time (as $s_{ij}+1\leq n$). Step-10 and Step-11 are $\mathcal{O}(1)$ processes. 
Checking $T=\phi$ or $T\neq \phi$ takes constant time. We repeat Step-5 to Step-12 $\binom{r}{2}$ times. 
Since one iteration of Step-5 to Step-12 can be done in $\mathcal{O}(n^3)$ time, $\binom{r}{2}$ iteration 
of Step-5 to Step-12 will take $\mathcal{O}(n^5)$ time (because $\binom{r}{2}\leq n^2$). We can choose two 
vertices out of $n$ in $\binom{n}{2}$ ways. Now for each pair of vertices, we have to check whether it belongs 
to the same $D^{pq}_{k}$ or not for some $k\in\{0,\ldots, s_{pq}\}$. There are $s_{pq}+1$ choices of $k$. Since 
$\binom{n}{2}< n^2$ and $s_{pq}+1\leq n$, total time required in Step-13 is $\mathcal{O}(n^3)$. At last, 
Step-14 takes $\mathcal{O}(n^2)$ time as $\vert V([G]_{pq})\vert=\vert V(G)\vert=n$ and $\vert E([G]_{pq})\vert\leq n^2$. 
Since the entire algorithm can be performed in $\mathcal{O}(n^5)$ time, it is a polynomial time algorithm.
\medskip

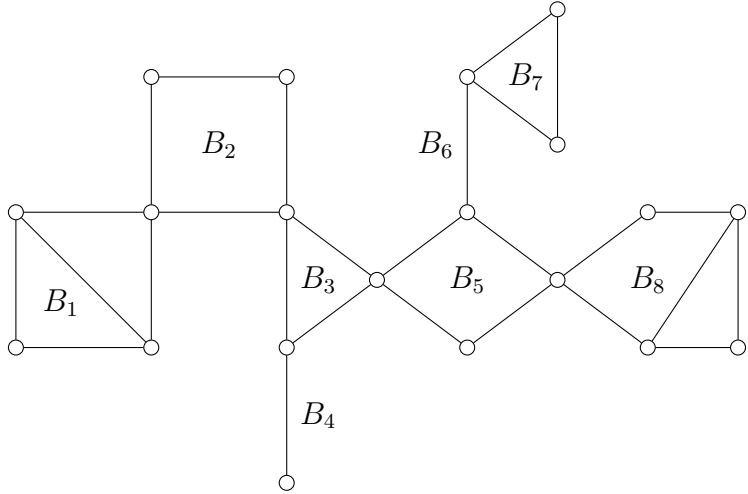
\begin{figure}[H]
	\centering
	\begin{tikzpicture}
  [scale=.6,auto=left,every node/.style={circle,scale=0.5}]
    
  \node[draw] (n1) at (0,0)  {};
  \node[draw] (n2) at (3,0)  {};
  \node[draw] (n3) at (0,3) {};
   \node[draw](n4) at (3,3) {};
  \node[draw] (n5) at (3,6) {};
  \node[draw] (n6) at (6,6) {};
  \node[draw] (n7) at (6,3) {};
  \node[draw] (n8) at (6,0) {};
  \node[draw] (n9) at (6,-3) {};
  \node[draw] (n10) at (8,1.5) {};
  
  \node[draw] (n11) at (10,3)  {};
  \node[draw] (n12) at (12,1.5)  {};
  \node[draw] (n13) at (10,0) {};
   \node[draw](n14) at (14,0) {};
  \node[draw] (n15) at (16,0) {};
  \node[draw] (n16) at (16,3) {};
  \node[draw] (n17) at (14,3) {};
  \node[draw] (n18) at (10,6) {};
  \node[draw] (n19) at (12,7.5) {};
  \node[draw] (n20) at (12,4.5) {};
  
  \node[scale=2] (m1) at (1,1) {$B_{1}$};
  \node[scale=2] (m2) at (4.5,4.5)  {$B_{2}$};
  \node[scale=2] (m3) at (6.7,1.5) {$B_{3}$};
   \node[scale=2](m4) at (6.7,-1.5) {$B_{4}$};
  \node[scale=2] (m5) at (10,1.5) {$B_{5}$};
  \node[scale=2] (m6) at (9.3,4.5) {$B_{6}$};
  \node[scale=2] (m7) at (11.3,6) {$B_{7}$};
  \node[scale=2] (m8) at (14,1.5) {$B_{8}$};

 \foreach \from/\to in {n1/n2,n1/n3, n2/n3, n2/n4, n3/n4, n4/n5, n5/n6, n6/n7, n4/n7, n7/n8, n8/n9, n8/n10, n7/n10, n10/n11, n11/n12, n12/n13, n10/n13, n12/n14, n14/n15, n15/n16, n16/n17, n12/n17, n14/n16, n11/n18, n18/n19, n19/n20, n18/n20}
    \draw[] (\from) -- (\to);
   
\end{tikzpicture}
	\caption{A graph $G$ such that $\mathcal{B}(G)$ is a tree}\label{figalg}
\end{figure}

\begin{example}{\rm
Consider the graph $G$ in the \Cref{figalg}. It is clear from the graph that $\mathcal{B}(G)$ is a tree. 
Now $\mathcal{P}=\{B_{1}, B_{4}, B_{7}, B_{8}\}$ is the set of blocks of $G$ containing only one cut 
vertex of $G$, i.e., $\mathcal{P}$ is the set of all pendant vertices of $\mathcal{B}(G)$. Consider 
$T$ as in Step-4 of the algorithm. Then
$$T=\{(B_{1}, B_{4}), (B_{1}, B_{7}), (B_{1},B_{8}), (B_{4},B_{7}), (B_{4},B_{8}), (B_{7}, B_{8})\}.$$

\noindent We have $\vert E(G)\vert =27$ and can observe that 
\medskip

\begin{enumerate}[$\bullet$]

\item $[G]_{14}=K_{4}\cup K_{4}\cup K_{13}\cup K_{2}$ and $d_{14}= \binom{4}{2}+\binom{4}{2}+\binom{13}{2}+\binom{2}{2}-27=64$;\medskip

\item $[G]_{17}=K_{4}\cup K_{4}\cup K_{4}\cup K_{8}\cup K_{2}\cup K_{3}$ and $d_{17}= \binom{4}{2}+\binom{4}{2}+\binom{4}{2}+\binom{8}{2}+\binom{2}{2}+\binom{3}{2}-27=23$;\medskip

\item $[G]_{18}=K_{4}\cup K_{4}\cup K_{4}\cup K_{7}\cup K_{5}$ and $d_{18}= \binom{4}{2}+\binom{4}{2}+\binom{4}{2}+\binom{7}{2}+\binom{5}{2}-27=22$;\medskip

\item $[G]_{47}=K_{2}\cup K_{9}\cup K_{8}\cup K_{2}\cup K_{3}$ and $d_{47}= \binom{2}{2}+\binom{9}{2}+\binom{8}{2}+\binom{2}{2}+\binom{3}{2}-27=42$;\medskip

\item $[G]_{48}=K_{2}\cup K_{9}\cup K_{7}\cup K_{5}$ and $d_{48}= \binom{2}{2}+\binom{9}{2}+\binom{7}{2}+\binom{5}{2}-27=48$;\medskip

\item $[G]_{78}=K_{3}\cup K_{2}\cup K_{13}\cup K_{5}$ and $d_{78}= \binom{3}{2}+\binom{2}{2}+\binom{13}{2}+\binom{5}{2}-27=65$.
\end{enumerate}
\medskip

\noindent Therefore, $\mathbf{CCM}(G)=d_{18}=22$, i.e., the minimum number of edges needed to add to 
$G$ to make it a closed graph with Cohen-Macaulay binomial edge ideal is $22$, and $[G]_{18}$ is the 
only minimum $\mathbf{CCM}$-completion of $G$.
}
\end{example}

$\mathbf{PIG}$-completion problem is NP-hard, but for $G\in\mathbf{BT}$, the $\mathbf{CCM}$-completion problem 
can be performed in polynomial time. Since $\mathbf{CCM}\subseteq \mathbf{PIG}$, the following question arise 
naturally.

\begin{question}\label{quesccm-comp}{\rm 
For an arbitrary graph $G$, can we give a polynomial-time algorithm to compute $\mathbf{CCM}(G)$ and a minimum $\mathbf{CCM}$-completion of $G$?
}
\end{question}

\section{Unmixed and Cohen-Macaulay Property of Binomial Edge Ideals of Subgraphs}\label{ccmsubgraph}

In this section, we discuss when unmixed and Cohen-Macaulay properties of binomial edge ideals are hereditary for subgraphs. Also, we study the accessible completion of graphs and Cohen-Macaulay criterion of binomial edge ideals of whisker graphs.

\begin{proposition}\label{um4}
Let $G$ be a graph such that $J_{G}$ is unmixed. Let 
$v\in V(G)$ be a free vertex in $\Delta(G)$, and $F$ 
the facet of $\Delta(G)$ with $v\in F$. If 
$F\setminus\lbrace v\rbrace\not\subseteq T$ for each $T\in\mathscr{C}(G)$, with $\vert T\vert\neq 1$, then $J_{G\setminus\lbrace v\rbrace}$ is unmixed.
\end{proposition}

\begin{proof}
We first prove that 
$\mathscr{C}(G\setminus\lbrace v\rbrace)\subseteq\mathscr{C}(G)$. 
Let $T\in\mathscr{C}(G\setminus\lbrace v\rbrace)$ and $u\in T$ be any vertex. 
We will show that $u$ is a cut point of $H=G\setminus (T\setminus\{ u\})$. Suppose that $u$ is not a cut point of $H$. Since  $u$ is a cut point of 
$H^{\prime}=G^{\prime}\setminus (T\setminus\{ u\})$, 
where $G^{\prime}=G\setminus\lbrace v\rbrace$, there exist two 
vertices $u_{1},u_{2}\in V(H^{\prime})$ such that all the paths 
from $u_{1}$ to $u_{2}$ pass through $u$. Moreover, 
there exists a path $\pi=u_{1},\ldots,v^{\prime},v,v^{\prime\prime},\ldots,u_{2}$
in $H$, that does not pass through $u$. Since $F$ is a clique and $\lbrace v^{\prime},v\rbrace,\lbrace v,v^{\prime\prime}\rbrace\in E(H)$, we have $\lbrace v^{\prime},v^{\prime\prime}\rbrace\in E(H)$ and so $\lbrace v^{\prime},v^{\prime\prime}\rbrace\in E(H^{\prime})$. Hence we can obtain a new path $\pi^{\prime}=u_{1},\ldots,v^{\prime},v^{\prime\prime},\ldots,u_{2}$ that is contained in $H^{\prime}$ 
and does not pass through $u$, a contradiction. Therefore $u$ is a cut point of $H$ and hence $T\in \mathscr{C}(G)$.\par 

\noindent\textbf{Case I.} Let $T\in\mathscr{C}(G\setminus\lbrace v\rbrace)\subseteq\mathscr{C}(G)$, such that $\vert T\vert\neq 1$. Let $c(T)$ be the number of connected components when we remove $T$ from $G$. Since $v$ is a free vertex, one of these components contains $v$ and let that component be $G^{v}$. So, in $(G\setminus\lbrace v\rbrace)\setminus T$, that component 
is $G^{v}\setminus\lbrace v\rbrace\neq\phi$ (since $F\setminus\lbrace v\rbrace\not\subseteq T$) and other components 
remain the same as in $G\setminus T$. Therefore, the number of 
connected components in $(G\setminus\{ v \})\setminus T$ is also $c(T)$. Since $J_{G}$ is unmixed, it follows from Lemma \ref{um3} that $c_{G\setminus\{v\}}(T)=\vert T\vert +1$.\par 

\noindent\textbf{Case II.} Let $T\in\mathscr{C}(G\setminus\lbrace v\rbrace)\subseteq\mathscr{C}(G)$, 
such that $\vert T\vert= 1$. Let 
$c(T)$ be the number of connected components in $G\setminus T$. Then from Lemma \ref{um3}, $c(T)=2$ 
as $J_{G}$ is unmixed. Let $T=\lbrace i\rbrace$. Then $i$ is a 
cut point of $(G\setminus\{ v\})\setminus (T\setminus\{ i\}) =G\setminus\lbrace v\rbrace$. So the number of connected components, 
after removing $T$ from $G\setminus\lbrace v\rbrace$, is at least $2$. 
Moreover, corresponding to each connected component of $(G\setminus \{v\})\setminus T$, we get a connected 
component in $G\setminus T$. Since $c(T)=2$, $c_{G\setminus\{v\}}(T)=2$. \par 
Now $G\setminus\{ v\}$ is connected as $v$ is a free vertex. Therefore, for all $T\in\mathscr{C}(G\setminus\lbrace v\rbrace)$, we have $c_{G\setminus\{v\}}(T)=\vert T\vert +1$. Hence, it follows from Lemma \ref{um3} 
that $J_{G\setminus\lbrace v\rbrace}$ is unmixed.
\end{proof}

\begin{proposition}\label{um5}
Let $G$ be a graph such that $J_{G}$ is unmixed and $v\in V(G)$. Let $G\setminus \lbrace v\rbrace$ be connected and $J_{G\setminus\lbrace v\rbrace}$ be unmixed. 
Then the following hold good:
\begin{enumerate}
\item[(i)] $v$ is a free vertex in $\Delta(G)$\\
and
\item[(ii)] $F\setminus\lbrace v\rbrace\not\subseteq T$ for each $T\in\mathscr{C}(G\setminus\{v\})$, where $F$ is the facet of $\Delta(G)$ with $v\in F$.
\end{enumerate}
\end{proposition}

\begin{proof}
Let $T\in\mathscr{C}(G)$ with $v\in T$. Since $J_{G}$ is unmixed, from Lemma \ref{um3}, we have $c_{G}(T)=\vert T\vert +1$. As $v\in T$, we have 
$G[\overline{T}]=(G\setminus\lbrace v\rbrace)[\overline{T}]$. 
So for each $i(\neq v)\in T$, we have 
$ G[\overline{T}\cup\{ i\}]=(G\setminus\{ v\})[\overline{T}\cup\{ i\}]$. Therefore $T\setminus\lbrace v\rbrace$ satisfies 
the cut point property for $G\setminus\lbrace v\rbrace$. Since 
$J_{G\setminus\lbrace v\rbrace}$ is unmixed and $G\setminus\{v\}$ is connected, the number of connected 
components in 
$(G\setminus\{ v\})\setminus (T\setminus\{ v\})=G\setminus T$ is $\vert T\setminus\lbrace v\rbrace\vert +1\neq \vert T\vert+1$, which is a contradiction. Therefore, 
$v\not\in T$ for each $T\in\mathscr{C}(G)$. Hence, by Proposition 
\ref{um1}, $v$ is a free vertex in $\Delta(G)$.
\par 

Let $F$ be the facet of $\Delta(G)$ with $v\in F$. Suppose 
$F\setminus\lbrace v\rbrace\subseteq T$ for some 
$T\in\mathscr{C}(G\setminus\{v\})$. We know $\mathscr{C}(G\setminus\{v\})\subseteq \mathscr{C}(G)$. Since $J_{G}$ is unmixed, the number of connected components in $G\setminus T$ is $c_{G}(T)=\vert T\vert +1$. One of these components is only $\lbrace v\rbrace$ as 
$v$ is a free vertex and $F\setminus\lbrace v\rbrace\subseteq T$. 
Therefore, the number of connected components in $(G\setminus\lbrace v\rbrace)\setminus T$ is $\vert T\vert$, which is a contradiction by Lemma \ref{um3} to the fact that $J_{G\setminus\lbrace v\rbrace}$ is unmixed.
\end{proof}

\begin{corollary}\label{um6}
Let $G$ be a graph with $v\in V(G)$, such that $v$ is a free vertex 
in $\Delta(G)$ and $F$ is the facet of $\Delta(G)$ containing $v$. 
Suppose that for all $T\in\mathscr{C}(G)$ we have $F\setminus\lbrace v\rbrace\not\subseteq T$, then the following conditions are equivalent:
\begin{enumerate}
\item[(a)] $J_{G}$ is unmixed;
\item[(b)] $J_{G\setminus\lbrace v\rbrace}$ is unmixed.
\end{enumerate}
\end{corollary}

\begin{proof}
From Lemma \ref{um2}, we have $\mathscr{C}(G)=\mathscr{C}(G\setminus\lbrace   v\rbrace)$. From the proof of Proposition \ref{um4}, we have that the 
number of connected components in $G[\overline{T}]$ is equal 
to the number of connected components in 
$(G\setminus\lbrace v\rbrace)[\overline{T}]$, for all 
$T\in \mathscr{C}(G)=\mathscr{C}(G\setminus\lbrace  v\rbrace)$. 
Since $v$ is a free vertex, we have 
$G\setminus \lbrace v\rbrace$ is connected and hence the proof 
follows from Lemma \ref{um3}.
\end{proof}

Suppose $G$ is a graph, such that $J_{G}$ is unmixed. From the above 
results, especially Corollary \ref{um6},   
we have seen that removing some special vertices from $G$, we get some induced subgraphs of $G$, whose corresponding binomial edge ideals are 
also unmixed. Therefore, a natural question arise: whether one 
can classify all such subgraphs of a graph $G$. The same question 
can be asked if we replace unmixed by Cohen-Macaulay. The rest of this section is devoted to the following classification problem: Given a graph $G$ that is closed with Cohen-Macaulay (respectively unmixed) $J_{G}$, classify all closed induced subgraphs of $G$ which have Cohen-Macaulay (respectively unmixed) binomial edge ideals.

\begin{corollary}\label{um7}
Let $G$ be a graph such that $J_{G}$ is Cohen-Macaulay. Suppose 
$G\setminus\lbrace v\rbrace$ is connected and $J_{G\setminus\lbrace v\rbrace}$ 
is Cohen-Macaulay. Then,
\begin{enumerate}
\item[(i)] $v$ is a free vertex in $\Delta(G)$\\
and 
\item[(ii)] $F\setminus\lbrace v\rbrace\not\subseteq T$, 
for each $T\in\mathscr{C}(G\setminus\{v\})$, where $F$ is the facet 
of $\Delta(G)$ containing $v$.
\end{enumerate}
\end{corollary}

\begin{proof}
$J_{G}$ is Cohen-Macaulay, hence unmixed. The 
proof follows from Proposition \ref{um5}.
\end{proof}

Next, we show that any induced subgraph of a closed graph (respectively $\mathbf{CCM}$ graph) is closed (respectively 
$\mathbf{CCM}$). We prove the results using the labelling of vertices with respect to which the graph is closed, 
and this helps us find a labelling of the subgraph with respect to which it is closed. However, 
it is possible to write a shorter proof if one uses the structure of closed graphs and 
$\mathbf{CCM}$ graphs discussed in Propositions \ref{closedpath}, \ref{closedindecom}, and Theorem \ref{closedcmstruc}.

\begin{lemma}\label{um8}
Subgraph of a closed graph is closed.
\end{lemma}

\begin{proof}
Let $G$ be a closed graph with respect to a given labelling on $[n]$. 
Let us remove the vertex $m\in V(G)$. Then the vertices of 
$G\setminus\lbrace m\rbrace$ are $1,2,\ldots,m-1,m+1,\ldots,n$. We 
rename these as $1,2,\ldots,m,m+1,\ldots,n-1$, i.e., for a vertex 
$k\in V(G)$, if $k<m$, then label it as $k$, and if $k>m$, 
then label it as $k-1$, in $G\setminus\lbrace m\rbrace$.\medskip

\noindent\textbf{Claim:} $G\setminus\lbrace m\rbrace$ is closed with respect to this labelling.\medskip

\noindent \textit{Proof of the claim.} Let 
$\lbrace i,k\rbrace\in E(G\setminus\lbrace m\rbrace)$, with $i<k$.\par 

\noindent\textbf{Case I.} \, Let $k<m$. Then $\lbrace i,k\rbrace\in E(G)$ implies 
that $\lbrace i,j\rbrace,\lbrace j,k\rbrace\in E(G)$, for all $i<j<k<m$,  
since $G$ is closed. Therefore, 
$\lbrace i,j\rbrace,\lbrace j,k\rbrace\in E(G\setminus\lbrace m\rbrace)$, 
for all $i<j<k$.\par

\noindent\textbf{Case II.} \, Let $i<m$ and $k\geq m$. Then the edge 
$\lbrace i,k\rbrace\in E(G\setminus\lbrace m\rbrace)$ is the same 
edge as $\lbrace i,k+1\rbrace\in E(G)$. Since $G$ is closed, 
$\lbrace i,k+1\rbrace\in E(G)$ implies 
$\lbrace i,j\rbrace,\lbrace j,k+1\rbrace\in E(G)$, for all $i<j<k+1$. 
Let $i<p<k$. If $p<m$ then $\lbrace i,p\rbrace,\lbrace p,k+1\rbrace\in E(G)$ 
and hence $\lbrace i,p\rbrace,\lbrace p,k\rbrace\in E(G\setminus\lbrace m\rbrace)$. Again, if $p\geq m$, then $i<p+1<k+1$ in $V(G)$ implies 
$\lbrace i,p+1\rbrace,\lbrace p+1,k+1\rbrace\in E(G)$ and hence 
$\lbrace i,p\rbrace,\lbrace p,k\rbrace\in E(G\setminus\lbrace m\rbrace)$ 
since $p+1>m$.\par

\noindent\textbf{Case III.}\, Let $i\geq m$. Then 
$\lbrace i,k\rbrace\in E(G\setminus\lbrace m\rbrace)$ 
is the same edge as $\lbrace i+1,k+1\rbrace\in E(G)$. 
Since $G$ is closed, $\lbrace i+1,k+1\rbrace\in E(G)$ implies that 
$\lbrace i+1,j+1\rbrace,\lbrace j+1,k+1\rbrace\in E(G)$, 
for all $m<i+1<j+1<k+1$ and hence $\lbrace i,j\rbrace,\lbrace j,k\rbrace\in E(G\setminus\lbrace m\rbrace)$, for all $i<j<k$. \par 

Combining all the cases we get that if $\lbrace i,k\rbrace\in E(G\setminus\lbrace m\rbrace)$, with $i<k$, then $\lbrace i,j\rbrace,\lbrace j,k\rbrace\in E(G\setminus\lbrace m\rbrace)$ for all $i<j<k$. Therefore, from Theorems  
\ref{intordering} and \ref{propintgraph}, it follows that $G\setminus\lbrace m\rbrace$ is closed. Since $m$ is an arbitrary vertex of $G$, any subgraph of $G$ is closed.
\end{proof}

\begin{theorem}\label{um9}
Let $G$ be a connected graph on $[n]$, which is closed with respect to the given labelling. If $J_{G}$ is Cohen-Macaulay, then $J_{H}$ is Cohen-Macaulay for any subgraph $H$ of $G$.
\end{theorem}

\begin{proof}
Let $m\in V(G)$ be any vertex. Remove $m$ and consider the graph 
$H=G\setminus\lbrace m\rbrace$, with the labeling mentioned in the 
Lemma \ref{um8}, for which $H$ is closed. Suppose 
$\lbrace i,j+1\rbrace$ with $i<j$, and $\lbrace j,k+1\rbrace$ with $j<k$, 
are edges of $H$. Then, to prove $J_{H}$ is Cohen-Macaulay we have to 
show that $\lbrace i,k+1\rbrace\in E(H)$.\par

\noindent\textbf{Case I ($i>m-1$).} In this case $m-1<i<j<k$. Therefore, 
$$\lbrace i,j+1\rbrace,\lbrace j,k+1\rbrace\in E(H) \Rightarrow \lbrace i+1,j+2\rbrace,\lbrace j+1,k+2\rbrace\in E(G).$$ 
By Proposition \ref{closedCM}, we have $\lbrace i+1,k+2\rbrace\in E(G)$, 
since $J_{G}$ is Cohen-Macaulay. Hence 
$\lbrace i,k+1\rbrace\in E(H)$, since $i+1>m$.\par 

\noindent\textbf{Case II ($i\leq m-1,\,j>m-1$).} In this case $m-1<j<k$. Therefore 
$$\lbrace i,j+1\rbrace,\lbrace j,k+1\rbrace\in E(H) \Rightarrow \lbrace i,j+2\rbrace,\lbrace j+1,k+2\rbrace\in E(G).$$ By Proposition \ref{closedCM}, we have $\lbrace i,k+2\rbrace\in E(G)$, since $J_{G}$ is Cohen-Macaulay. 
Therefore, $\lbrace i,k+1\rbrace\in E(H)$, since $i\leq m-1$ and $k+2>m$.\par

\noindent\textbf{Case III ($i<m-1,\,j= m-1$).} In this case $m-1=j<k$. Therefore 
$$\lbrace i,j+1\rbrace,\lbrace j,k+1\rbrace\in E(H) \Rightarrow  \lbrace i,j+2\rbrace,\lbrace j,k+2\rbrace\in E(G).$$ 
$G$ is closed, and therefore, $\lbrace i,j+2\rbrace\in E(G)$ 
implies $\lbrace i,j+1\rbrace\in E(G)$. 
By Proposition \ref{closedCM}, we have $\lbrace i,k+2\rbrace\in E(G)$, 
since $J_{G}$ is Cohen-Macaulay. Hence, $\lbrace i,k+1\rbrace\in E(H)$, 
since $k+2>m$.\par

\noindent\textbf{Case IV ($i<j< m-1,\,k\geq m-1$).} In this case, $$\lbrace i,j+1\rbrace,\lbrace j,k+1\rbrace\in E(H) \Rightarrow \lbrace i,j+1\rbrace,\lbrace j,k+2\rbrace\in E(G).$$ 
By Proposition \ref{closedCM}, we have $\lbrace i,k+2\rbrace\in E(G)$ , 
since $J_{G}$ is Cohen-Macaulay. Therefore, $\lbrace i,k+1\rbrace\in E(H)$, 
since $k+2>m$.\par

\noindent\textbf{Case V ($i<j<k<m-1$).} In this case, $$\lbrace i,j+1\rbrace,\lbrace j,k+1\rbrace\in E(H) \Rightarrow \lbrace i,j+1\rbrace,\lbrace j,k+1\rbrace\in E(G).$$ By Proposition \ref{closedCM}, we have 
$\lbrace i,k+1\rbrace\in E(G)$, since $J_{G}$ is Cohen-Macaulay. Therefore, 
$\lbrace i,k+1\rbrace\in E(H)$, since $k+1<m$.\par

Combining all the cases, we get that if $\lbrace i,j+1\rbrace$ with $i<j$, 
and $\lbrace j,k+1\rbrace$ with $j<k$, are edges of $H$ then 
$\lbrace i,k+1\rbrace$ is  an edge of $H$. Hence, from 
Proposition \ref{closedCM}, it follows that $J_{H}$ is Cohen-Macaulay. 
Therefore, $m$ being an arbitrary vertex, 
$J_{H}$ is Cohen-Macaulay for every subgraph $H$ of $G$.
\end{proof}

Let $G$ be a graph and $v$ be a vertex of $G$. We define the graph $G_{v}$ as follows:
\begin{enumerate}
\item[$\bullet$] $V(G_{v})=V(G),$
\item[$\bullet$] $E(G_{v})=E(G)\cup \{\{i,j\}\mid i,j\in \mathcal{N}_{G}(v)\,\,\text{with}\,\, i\neq j\}.$
\end{enumerate}
\medskip

In \cite{bms1}, Bolognini et al. have introduced the concept of and 
have proved that for a graph $G$, Cohen-Macaulay property of $J_{G}$ implies the 
accessible property of $G$. They have conjectured about the converse in \cite[Conjecture 1.1]{bms1}. 
The following Proposition \ref{acc-construct} gives a technique to make an unmixed 
binomial edge ideal accessible by adding some special edges. Let us denote the class 
of accessible graphs by $\mathbf{ACC}$,

\begin{proposition}\label{acc-construct}
Let $G$ be a graph such that $J_{G}$ is unmixed. 
We get an accessible graph $H$ with $V(H)=V(G)$, by adding some special edges to $G$.
\end{proposition}

\begin{proof}
If $G$ is accessible, then consider $H=G$. Suppose $G$ is not accessible, then 
there exists $T_{1}\in \mathscr{C}(G)$ such that $T$ is not accessible. Take any 
$v_{1}\in T_{1}$ and consider the graph $G_{v_{1}}$. By \cite[Lemma 4.5]{bms1}, 
$\mathscr{C}(G_{v_{1}})=\{T\in \mathscr{C}(G)\mid v_{1}\not\in T\}$ and 
$J_{G_{v_{1}}}$ is unmixed. If $G_{v_{1}}$ is accessible, then consider 
$H=G_{v_{1}}$ and if not accessible, then repeat the process. Due to 
\cite[Lemma 4.5]{bms1}, after finite steps, we get a graph 
$H=(\ldots((G_{v_{1}})_{v_{2}})\ldots)_{v_{r}}$, such that $H$ is accessible.
\end{proof}

\begin{figure}[H]
\centering
\begin{subfigure}{0.45\textwidth}
\begin{tikzpicture}
  [scale=.6,auto=left,every node/.style={circle,scale=0.5}]
 
  \node[draw] (n1) at (-1,3)  {$1$};
  \node[draw] (n2) at (0,0)  {$2$};
  \node[draw] (n3) at (1.5,3) {$3$};
   \node[draw] (n4) at (3,-1) {$4$};
   \node[draw] (n5) at (4.5,3) {$5$};
  \node[draw] (n6) at (6,0) {$6$};
  \node[draw] (n7) at (7,3) {$7$};

  \foreach \from/\to in {n1/n2,n2/n3,n3/n4,n4/n5, n5/n6, n6/n7, n3/n6, n2/n5}
    \draw[] (\from) -- (\to);
   
\end{tikzpicture}
\caption{A non-accessible Graph $G$}\label{fignon-acc}
\end{subfigure}
\begin{subfigure}{0.45\textwidth}
\begin{tikzpicture}
  [scale=.6,auto=left,every node/.style={circle,scale=0.5}]
 
  \node[draw] (n1) at (-1,3)  {$1$};
  \node[draw] (n2) at (0,0)  {$2$};
  \node[draw] (n3) at (1.5,3) {$3$};
   \node[draw] (n4) at (3,-1) {$4$};
   \node[draw] (n5) at (4.5,3) {$5$};
  \node[draw] (n6) at (6,0) {$6$};
  \node[draw] (n7) at (7,3) {$7$};

  \foreach \from/\to in {n1/n2,n2/n3,n3/n4,n4/n5, n5/n6, n6/n7, n3/n6, n2/n5, n2/n4, n2/n6, n4/n6}
    \draw[] (\from) -- (\to);
   
\end{tikzpicture}

\caption{The graph $G_{3}$, which is accessible}\label{figacc}
	\end{subfigure}

	\caption{A non-accessible graph $G$ with unmixed $J_{G}$ and the accessible 
	graph $G_{3}$ with Cohen-macaulay $J_{G_{3}}$.}\label{figglu}
\end{figure}
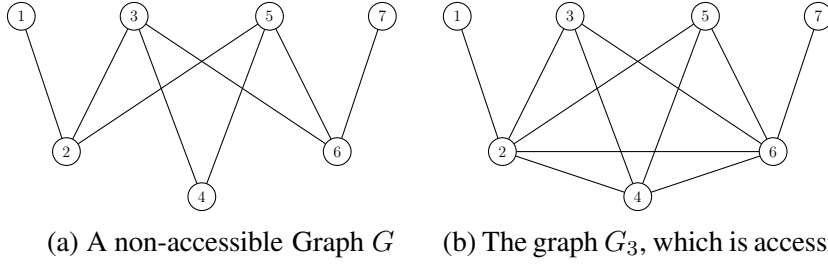

\begin{example}{\rm
Consider the graph $G$ in Figure (2a). It was shown in 
\cite[Example 2.3]{bms1} that $J_{G}$ is unmixed but $G$ is not accessible. In fact, 
$\mathscr{C}(G)=\{\phi, \{2\},\{6\}, \{2,6\},\{3,5\}, \{2,4,6\}\}$. As in the proof 
of Proposition \ref{acc-construct}, we choose the non-accessible cutset $\{3,5\}$ of 
$G$ and take $3\in \{3,5\}$. Now consider the graph $G_{3}$ (see Figure 2(b)). 
Then $J_{G_{3}}$ is unmixed and $\mathscr{C}(G_{3})=\{\phi, \{2\},\{6\}, \{2,6\}, \{2,4,6\}\}$. 
Thus, $G_{3}$ is accessible and using Macaulay2, we get $J_{G_{3}}$ is Cohen-Macaulay.
}
\end{example}

In Proposition \ref{acc-construct}, we see that 
for a graph $G$ with unmixed $J_{G}$, there exists an $\mathbf{ACC}$-completion 
of $G$. Thus, we propose the following question:

\begin{question}\label{quesacc-comp}{\rm 
For a graph $G$ with unmixed $J_{G}$, is the method discussed in 
Proposition \ref{acc-construct} the only way to get a $\mathbf{ACC}$-completion of 
$G$? Also, what will be $\mathbf{ACC}(G)$? Can we find a polynomial-time algorithm 
to compute $\mathbf{ACC}(G)$?
}
\end{question}

Let $G$ be a graph with $V(G)=\{x_{1},\ldots,x_{n}\}$. The \textit{whisker graph} or \textit{suspension} of $G$, denoted by $W_{G}$, is the graph attaching $n$ new vertices $\{y_{1},\ldots,y_{n}\}$ to $G$ as follows:
\begin{enumerate}
\item[$\bullet$] $V(W_{G})=\{x_{1},\ldots,x_{n},y_{1},\ldots,y_{n}\},$
\item[$\bullet$] $E(W_{G})=E(G)\cup \{\{x_{i},y_{i}\}\mid i=1,\ldots,n\}.$
\end{enumerate}
\medskip

In the following Theorem \ref{whiskercm}, we give sufficient and necessary 
conditions for a whisker graph to be Cohen-Macaulay.

\begin{theorem}\label{whiskercm}
Let $G$ be a graph. Then the following are equivalent.
\begin{enumerate}[(i)]
\item $J_{W_{G}}$ is Cohen-Macaulay.
\item $J_{W_{G}}$ is unmixed.
\item $G$ is complete.
\end{enumerate}
\end{theorem}

\begin{proof}
(i) $\Rightarrow$ (ii) is well known. 
\smallskip

\noindent(ii) $\Rightarrow$ (iii): Let $V(G)=\{x_{1},\ldots,x_{n}\}$ 
and $V(W_{G})=V(G)\cup \{y_{1},\ldots, y_{n}\}$ be such that $\{x_{i},y_{i}\}\in E(W_{G})$ 
for $1\leq i\leq n$. Suppose $G$ is not complete. Then there exist $x_{j},x_{k}\in V(G)$ 
such that $\{x_{j},x_{k}\}\not\in E(G)$. Now consider $T=\mathcal{N}_{G}(x_{j})=\{x_{j_{1}},\ldots,x_{j_{s}}\}$. Then each $x\in T$ is a cut vertex of $W_{G}\setminus (T\setminus \{x\})$ 
and hence, $T\in \mathscr{C}(W_{G})$. Note that $W_{G}\setminus T$ has 
$\{y_{j_{1}}\},\ldots, \{y_{j_{s}}\}$ as connected components, with one containing  
$x_{j}$ and one containing $x_{k}$. Therefore, $c_{W_{G}}(T)>\vert T\vert +1$, 
a contradiction to the fact that $J_{W_{G}}$ is unmixed. Hence $G$ is complete.
\smallskip

\noindent(iii) $\Rightarrow$ (i): Let $G$ be complete. Then $\mathcal{B}(W_{G})$ is a tree and every block of $W_{G}$ is complete. Thus, $J_{W_{G}}$ is Cohen-Macaulay by \cite[Theorem 1.1]{ehh}.
\end{proof}

\section*{Acknowledgements}
The authors wish to thank Neeldhara Misra and Saraswati Nanoti for their valuable 
help in understanding the complexity problem. 

\bibliographystyle{amsalpha}

\begin{thebibliography}{A}

\bibitem{mont} $\grave{\text{A}}$lvarez Montaner, J. \emph{Local cohomology of binomial edge ideals and their generic initial ideals}, Collect. Math. 71: no. 2, 331--348 (2020).

\bibitem{bms} Bolognini, D., Macchia, A., Strazzanti, F. \emph{Binomial edge ideals of bipartite graphs}, European J. Combin., Vol. 70, pp. 1--25 (2018).

\bibitem{bms1} Bolognini, D., Macchia, A., Strazzanti, F. \emph{Cohen-Macaulay binomial edge
ideals and accessible graphs}, J. Algebraic Combin. 55: no. 4, 1139--1170 (2022).

\bibitem{ch94} Conca, A., Herzog, J. \emph{On the Hilbert function of determinantal rings and their canonical module}, Proc. Amer. Math. Soc., 122(3): 677--681 (1994).

\bibitem{closed1} Cox, David A., Erskine, A. \emph{On closed graphs I}, Ars Combin. 120, pp. 259--274 (2015).

\bibitem{dross21} Dross, F., Hilaire, C., Koch, I., Leoni, V., Pardal, N., Lopez Pujato, M. I., Santos, V.F. \emph{On the proper interval completion problem within some chordal subclasses}, arXiv: 2110.07706 [cs.DM], 13 pp (2021), https://arxiv.org/abs/2110.07706.

\bibitem{ehh} Ene, V., Herzog, J., Hibi, T. \emph{Cohen-Macaulay binomial edge ideals}, Nagoya
Math. J., Vol. 204, pp. 57--68 (2011).

\bibitem{F} Faridi, S. \emph{Cohen-Macaulay properties of square-free monomial ideals}, J. Combin., Theory Ser. A 109(2), 299--329 (2005).

\bibitem{gks94} Golumbic, M. C., Kaplan, H., Shamir, R. \emph{On the complexity of DNA physical mapping} Adv. in Appl. Math. 15: no. 3, 251--261 (1994).

\bibitem{hhhrkara} Herzog, J., Hibi, T., Hreinsdottir, F., Kahle, T., Rauh, J. 
\emph{Binomial edge ideals and conditional independence statements}, 
Advances in Applied Mathematics 45:317--333 (2010).

\bibitem{hho} Herzog, J., Hibi, T. and Ohsugi, H. \emph{Binomial Ideals}, 
Springer, 321 p., (2018).

\bibitem{hr18} Herzog, J., Rinaldo, G. \emph{On the extremal Betti numbers of binomial edge ideals of block graphs}, Electron. J. Combin. 25, no. 1, Paper No. 1.63, 10 pp (2018).

\bibitem{matcomp} Hogben, L. \emph{Graph theoretic methods for matrix completion problems}, Linear Algebra Appl. 328: no. 1-3, 161--202 (2001).

\bibitem{kaplan99} Kaplan, H., Shamir, R., Tarjan, R. E. \emph{Tractability of parameterized completion problems on chordal, strongly chordal, and proper interval graphs}, SIAM J. Comput. 28: no. 5, 1906--1922 (1999).

\bibitem{km12} Kiani, D., Saeedi Madani, S. \emph{Binomial edge ideals of graphs}, Electron. J. Combin. 19: no. 2, Paper 44, 6 pp (2012).

\bibitem{km} Kiani, D., Saeedi Madani, S. \emph{Some Cohen-Macaulay and unmixed binomial
edge ideals}, Comm. Alg., Vol. 43, 12, pp. 5434--5453 (2015).

\bibitem{netcomp} Kim, M., Leskovec, J. \emph{The network completion problem: Inferring missing nodes and edges in networks}, Proceedings of the 11th International Conference on Data Mining, 47--58, (2011).

\bibitem{bfsdfs} Kozen, D. C. \emph{Depth-First and Breadth-First Search}, The Design and Analysis of Algorithms, Texts and Monographs in Computer Science, Springer, New York, 19--24 (1992).

\bibitem{ks19} Kumar, A., Sarkar, R. \emph{Hilbert series of binomial edge ideals}, Comm. Algebra 47: no. 9, 3830--3841 (2019).

\bibitem{lmrr} Lerda, A., Mascia, C., Rinaldo, G., Romeo, F. \emph{(S2)-condition and CohenMacaulay binomial edge ideals}, https://arxiv.org/abs/2107.04539, 2021.


\bibitem{lo} Looges, P.J., Olariu, S. \emph{Optimal greedy 
algorithms for indifference graphs}, 
TComput. Math. Appl. 25, 15--25 (1993).

\bibitem{mm13} Matsuda, K.; Murai, S. \emph{Regularity bounds for binomial edge ideals}, J. Commut. Algebra 5: no. 1, 141--149 (2013). 

\bibitem{ms} Mohammadi, F., Sharifan, L. \emph{Hilbert function of binomial edge ideals},
Comm. Algebra 42, Vol. 2, 688--703 (2014).

\bibitem{oh} Ohtani, M. \emph{Graphs and ideals generated by some 2-minors}, Comm. Algebra, 39: 905--917 (2011).

\bibitem{paton} Paton, K. \emph{An algorithm for the blocks and cutnodes of a graph}, Commun. ACM, 14: 468--475 (1971).

\bibitem{rst06} Rapaport, I., Suchan, K., Todinca, I. \emph{Minimal proper interval completions. Graph-theoretic concepts in computer science}, Lecture Notes in Comput. Sci., 4271, Springer, Berlin, 217--228 (2006).

\bibitem{raufrin} Rauf, A., Rinaldo, G. \emph{Construction of Cohen-Macaulay 
binomial edge ideals}, Communications in Algebra 42: 238--252 (2014).

\bibitem{cactus} Rinaldo, G. \emph{Cohen-Macaulay binomial edge ideals of cactus graphs}, J.
Algebra Appl. 18: no.4, 18 pp(2019).

\bibitem{rin_smaldev} Rinaldo, G., \emph{Cohen-Macauley binomial edge ideals of small deviation}, Bull. Math. Soc. Sci. Math. Roumanie (N.S.) 56(104): no.4, 497--503 (2013).

\bibitem{vi} Villarreal, R. H. \emph{Cohen-Macaulay graphs}, Manuscripta Math., 66: 277--293 (1990).

\bibitem{vil} Villarreal, R. H. \emph{Monomial Algebras, Monographs and Textbooks in Pure and Applied Mathematics} 238, Marcel Dekker, New York, (2001).

\bibitem{zz00} Zaare-Nahandi, R., Zaare-Nahandi, R. \emph{Gröbner basis and free resolution of the ideal of 2-minors of a 2×n matrix of linear forms}, Comm. Algebra 28: no. 9, 4433--4453 (2000).

\end{thebibliography}

\end{document}